%% file: main.tex
\documentclass{article}
\usepackage[utf8]{inputenc}
\usepackage{ytableau}
\usepackage{amsmath, amssymb, amsfonts, comment}
\usepackage{mathtools}
\usepackage{amsthm}
\usepackage{xcolor}
\usepackage{tabu}
\usepackage{array}
\usepackage[final]{showlabels}
\usepackage{url}

\usepackage[style = alphabetic, giveninits]{biblatex}
\usepackage{cases}[amsstyle]

\title{Schubert cells of mixed type in complex Lagrangian Grassmannians}
\author{Hyunmoon Kim}

\newif\ifdebug                                                      %
\debugfalse

\def\C{\mathbb{C}}
\def\R{\mathbb{R}}
\def\Z{\mathbb{Z}}
\def\Lag{\operatorname{Lag}}
\def\Sp{\operatorname{Sp}}
\def\Id{\mathrm{Id}}
\def\vv{\mathbf{v}}
\def\ww{\mathbf{w}}
\def\uu{\mathbf{u}}
\def\ee{\mathbf{e}}
\def\ff{\mathbf{f}}
\def\span{\operatorname{span}}
\def\pr{\operatorname{pr}}

\newtheorem{theorem}{Theorem}
\newtheorem{proposition}[theorem]{Proposition}
\newtheorem{remark}[theorem]{Remark}

\newtheorem{lemma}[theorem]{Lemma}
\newtheorem{example}[theorem]{Example}
\newtheorem{corollary}[theorem]{Corollary}
\newtheorem*{one}{Theorem}
\newtheorem*{two}{Theorem}

\usepackage{graphicx} 
\graphicspath{{\main/images/}{images/}}

\begin{document}

\maketitle
\begin{abstract}
We describe CW decompositions of complex Lagrangian Grassmannians, that contain as subcomplexes, CW decompositions of real Lagrangian Grassmannians by Schubert-Arnol'd cells. The degrees of attaching maps are explicitly computed in terms of quantities that can be read off from the corresponding shifted Young diagrams of mixed type. The signs are determined by a choice of lexicographical ordering on coordinates. As an immediate consequence, we obtain the homotopy extension property for real and complex Lagrangian Grassmannians. We also show some torsion classes in the integral homology of the real Lagrangian Grassmannian are contractible inside the complex Lagrangian Grassmannian.
\end{abstract}


\section{Introduction}
This paper solves a problem in algebraic topology motivated by geometric quantization. A key ingredient in geometric quantization is given by polarizations, which can loosely be described as some information to keep track of choices of observables that (Poisson) commute. At a linear level, these choices can be viewed as points on complex (or real) Lagrangian Grassmannians. As these Grassmannians have nontrivial topology, topologically nontrivial families of such choices exist. A famous example is the Maslov cycle. Its relevance in quantization was first observed by Maslov (cf Section II.2.2 \cite{Maslov}) and understood by Arnol'd in \cite{Arnold} in terms of the geometry of the real Lagrangian Grassmannian.

For real Lagrangian Grassmannians, higher dimensional homology classes were described by Arnol'd in his seminar (see footnote in Fuks \cite{FuchsMA}) and the cohomology ring structures were studied by Borel \cite{Borel} (Proposition 31.4, for the oriented Lagrangian Grassmannian), \cite{BorelMod2} (Theorem 12.1), and Fuks \cite{FuchsMA}. The torsion classes are all $2$-torsion, the $\Z_2$ cohomology ring is an exterior algebra, and the integral cohomology ring modulo torsion is another exterior algebra (see Section 22 of Vassilyev \cite{Vassilyev} for an exposition). The cellular decompositions into Schubert-Arnol'd cells were described Fuchs \cite{FuchsClassical}, with the incidence coefficients specified up to sign. The signs were recently computed explicitly in Rabelo \cite{Rabelo} from Lie theoretic techniques. 

The cohomology ring of complex Lagrangian Grassmannians were also studied by Borel \cite{Borel} (Theorem 26.1). Complex Lagrangian Grassmannians also have cellular decompositions into complex Schubert cells. These cells appear in only even degrees, so all attaching maps have degree zero, and there are no torsion classes in integral homology. These structures were described by Borel in \cite{BorelProc} (Theorem 3), Bernstein-Gel'fand-Gel'fand in \cite{BGG} (Theorem 2) and revisited by Pragacz in \cite{Pragacz} in the language of shifted Young diagrams, together with a purely combinatorial description of the integral cohomology ring. 

While these CW complexes are concrete combinatorial models for real and complex Lagrangian Grassmannians, they are insufficient to describe the induced map of chain complexes for the natural embedding of real Lagrangian Grassmannians inside complex Lagrangian Grassmannians. We construct CW decompositions of complex Lagrangian Grassmannians, that contain the CW complexes of real Lagrangian Grassmannians by real Schubert cells (Schubert-Arnol'd cells) as subcomplexes. They are subdivisions of the CW structure given by complex Schubert cells. Here are precise statements. 

\begin{one}
\label{thm:stratification}
There is a partition of the complex Lagrangian Grassmannian of $\R^{2n}$
\[ \Lag^\C(\R^{2n}) = \bigsqcup_{\lambda, \mu : \mu \le \lambda} C_{\lambda + i \mu} \]
where $\lambda$, $\mu$ are shifted Young diagrams associated to subsets of $\{1, 2, ..., n\}$ such that the following holds:
\begin{enumerate}
\item Each $C_{\lambda + i \mu}$ is diffeomorphic to a product of finitely many copies of $\R$, $\R^\times$, and $\C$, of real dimension $|\lambda| + |\mu|$.
\item The partition is a stratification with the following frontier condition:
\[ C_{\lambda+ i \mu} \subseteq \overline{C_{\lambda'+i\mu'}} \iff C_{\lambda +i \mu} \cap \overline{C_{\lambda' + i\mu'}}\neq \phi \iff \lambda \le \lambda' \text{ and } \mu \le \mu'.\]
\item When $\mu$ is the empty diagram, $C_{\lambda +i\mu}$  is a real Schubert cell $C_{\lambda}^{\R}$ of the embedded image of $\Lag^{\R}(\R^{2n})$ in $\Lag^{\C}(\R^{2n})$ by $\cdot \otimes_\R \C$. 
\item Each complex Schubert cell $C_\lambda$ of $\Lag^\C(\R^{2n})$ partitions into $\bigsqcup_{\mu: \mu \le \lambda} C_{\lambda + i\mu}$.
\end{enumerate}
\end{one}

\begin{two}
\label{thm:cellcomplex}
Let $\{C_{\lambda + i \mu, \varepsilon}\}_{\varepsilon \in \mathcal{E}(\mu)}$ be the connected components of $C_{\lambda + i \mu}$. There is a CW decomposition of $\Lag^\C(\R^{2n})$ having the $C_{\lambda + i \mu, \varepsilon}$'s as cells and a closed formula for the degrees for the attaching maps $\partial C_{\lambda + i \mu, \varepsilon} \to C_{\lambda' + i \mu', \varepsilon'}$ (equations \eqref{eq:bdy1}, \eqref{eq:bdy2}, \eqref{eq:bdy3}).
\end{two}

By a CW complex, we refer to the definitions of Sections 4, 5 of Whitehead \cite{Whitehead}. In particular, these definitions require attaching maps (or characteristic maps) to be realized by actual maps on the boundaries of balls. We note that in the literature the terms cell, cellular, CW complexes are often used in a weaker sense (eg Borel \cite{BorelProc}), or only the incidence coefficients are provided (eg Fuchs \cite{FuchsClassical}, Rabelo \cite{Rabelo}). We provide an explicit construction of the attaching maps in Proposition \ref{prop:attachingmaps}.

We call the cells of this complex \emph{Schubert cells of mixed type.} They are indexed by shifted Young diagrams of Pragacz \cite{Pragacz} with additional labels on the boxes. The degrees of the attaching maps are obtained using a straightforward, but involved computation in row reduction. It agrees up to sign with Fuchs \cite{FuchsClassical} and Rabelo \cite{Rabelo} (and has a different sign convention from \cite{Rabelo}). Some low dimensional examples are given in Figures \ref{fig:CP1} and \ref{fig:LagR4}.

An immediate consequence of this realization is that we have the homotopy extension property for the pair $(\Lag^\C(\R^{2n}), (\cdot \otimes_\R \C)(\Lag^\R(\R^{2n})))$. This construction is compatible with adding more variables (Corollary \ref{cor:HEP}). In fact, we can realize $\Lag^\R(\R^{2n})$ as a subcomplex of any $\Lag^\C(\R^{2(n+m)})$, and with all inclusion maps being compatible.

This construction is relevant in geometric quantization for the following reason. Geometric quantization requires a passage to the complex numbers, as observables are quantized from real valued smooth functions on symplectic manifolds to complex linear Hermitian operators on a Hilbert space. Locally, this task is performed by tensoring with $\C$. Our construction enables us to understand how some topologically nontrivial families of real Lagrangian subspaces behave after complexification. We refer to Corollary \ref{cor:nullhomotopic} for some examples.

In Section \ref{sec:schubert} we revisit the classical Schubert (Schubert-Arnol'd) calculus for real and complex Lagrangian Grassmannians in the language of shifted Young diagrams. In Section \ref{sec:mixed} we describe the Schubert calculus for Schubert cells of mixed type.

\input{Notation}
\begin{figure}
\centering \includegraphics[width = 3cm]{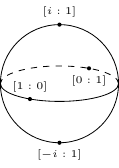}
\ytableausetup{centertableaux, boxsize = 5pt}
\caption{$\Lag^{\C}(\R^2)$ can be identified with the complex projective line. In this case the Schubert cells of mixed type are $C_{\ytableaushort{+}} = \{ [x + iy: 1] : y > 0\}$, $C_{\ytableaushort{-}} = \{ [x + iy: 1] : y < 0\}$, $C_{\ytableaushort{\circ}} = \{ [x: 1]\}$, and $C_{\phi} = \{ [1:0]\}$.}
\label{fig:CP1}
\end{figure}

\begin{figure}
\centering \includegraphics{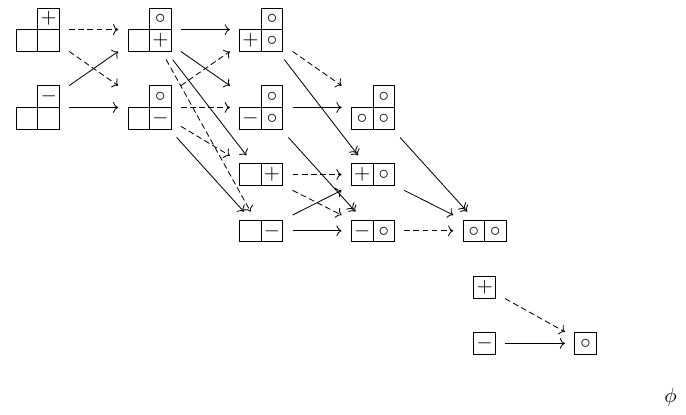}
\caption{A diagrammatic description of the CW decomposition of $\Lag^\C(\R^4)$ by Schubert cells of mixed type. An empty box, a box with a circle, a box with a $\pm$ denotes respectively, a copy of $\C$, a copy of $\R$, and a copy of $\R \times i\R^{\pm}$. Single headed arrows represent attaching maps of degree $\pm 1$, double headed arrows represent attaching maps of degree $\pm 2$, and filled, dashed arrows represent, respectively attaching maps with positive, negative signs. No arrows means either the cells are not incident or that they attach with degree $0$. The CW decomposition of $\Lag^{\R}(\R^4)$ is the subcomplex consisting of rightmost diagrams in every row except the first row, and the CW decomposition of $\Lag^{\C}(\R^2)$ is in the bottom right corner.}
\label{fig:LagR4}
\end{figure}
\input{Schubert}

\input{Mixed}

\input{Applications}


\printbibliography


{\ifdebug {\newpage} \else {\end{document}} \fi}

\medskip\noindent\textbf{To do:}
\begin{itemize}
\item blah blah
\item blah blah
\end{itemize}

\medskip\noindent\textbf{To discuss:}
\begin{itemize}
\item Cell complex / cell decomposition / cell structure
\item CW complex / CW decomposition / CW structure
\item Cellular complex / cellular decomposition / cellular structure

\item Lei Liu Borel FuksMA IkedaNaruse
\end{itemize}

\medskip\noindent\textbf{Conventions:}
\begin{itemize}
\item blah blah
\item blah blah
\end{itemize}

\medskip\noindent\textbf{Not for action:}
\begin{itemize}
\item blah blah
\item blah blah
\end{itemize}

\medskip\noindent\textbf{Journals to submit to}
\begin{itemize}
    \item Arnold Mathematical Journal (Dmitry Fuchs)
    \item Journal of symplectic geometry (
    \item Advances in Geometry (Kaoru Ono)
\end{itemize}

\end{document}

%% file: Notation.tex
\subsection{Notation}
\label{subsec:notation}

Let $\omega$ be a nondegenerate antisymmetric $\R$-bilinear form on $\R^{2n}$, and we will identify $ \R^{2n} \otimes_\R \C \cong \C^{2n}$ with $\vv\otimes_\R a \mapsto a\vv$. We will denote the $\C$ bilinear extension of $\omega$ to $\C^{2n}$ by $\omega^\C$. A real Lagrangian subspace $L$ of $\R^{2n}$ is an $n$ dimensional real vector subspace of $\R^{2n}$ such that the restriction of $\omega$ to $L \times L$ vanishes. A complex Lagrangian subspace $\Gamma$ of $\C^{2n}$ is an $n$ dimensional complex vector subspace of $\C^{2n}$ such that the restriction of $\omega^\C$ to $\Gamma \times \Gamma$ vanishes. We will refer to the Grassmannian manifold (cf Section 3.3 of Arnol'd \cite{Arnold}) consisting of all real Lagrangian subspaces of $\R^{2n}$ by $\Lag^{\R}(\R^{2n})$ and the Grassmannian manifold consisting of all complex Lagrangian subspaces of $\C^{2n}$ as $\Lag^{\C}(\R^{2n})$.

%% file: Schubert.tex
\section{Review of Schubert cells}
\label{sec:schubert}

\subsection{Shifted Young diagrams of Schubert cells}
\label{subsec:SYD}
For the real Lagrangian Grassmannian, their Schubert stratification has been developed independently by Arnol'd (cf \cite{FuchsClassical}). We will review the Schubert cells of the complex Lagrangian Grassmannian following Pragacz \cite{Pragacz}, with their indexing by shifted Young diagrams.

There are $2^n$ Schubert cells of $\Lag^\C(\R^{2n})$, corresponding to the $2^n$ subsets $J$ of $[n]:= \{ 1, 2, \cdots, n\}$.

The reduced row echelon forms of these Schubert cells are described as follows. We will fix a Darboux basis of $\R^{2n}$, $\{\ee_1, \cdots, \ee_n, \ff_1, \cdots, \ff_n\}$, so that 
\[ \omega(\ee_j, \ee_k) = \omega(\ff_j, \ff_k) = 0\quad \text{and} \quad \omega(\ee_j, \ff_k) = \begin{cases} 1 & \mbox{if } j = k \\ 0 & \mbox{otherwise.}\end{cases}\]

We will denote
\[ \mathbf{v} = \sum_{j = 1}^n q_j \ee_j + \sum_{k = 1}^n p_k \ff_k \in \C^{2n}\]
as a row vector, with the $p_k$'s in reverse order:
\[ [\mathbf{v}]_{\{\ee, \ff\}} := \left[ \begin{array}{ccc|ccc} q_1 & \cdots & q_n & p_n & \cdots & p_1 \end{array} \right].\]

If $\Gamma$ is a complex Lagrangian subspace of $\C^{2n}$, there is a unique $n \times 2n$ complex matrix in reduced row echelon form corresponding to $\Gamma$ in the fixed Darboux basis. Since $\Gamma$ is Lagrangian, $\ee_k$ and $\ff_k$ cannot both be pivot locations for two different row vectors of the reduced row echelon form. Let $I:= [n] \backslash J$. Then denote by $j_k$ the $k$th smallest element of $J$ and by $i_\ell$ the $\ell$th smallest element of $I$ so that $J = \{ j_k\}_{k = 1}^{|J|}$ and $I = \{ i_\ell\}_{\ell = 1}^{|I|}$. Then the reduced row echelon form corresponding to $J$ is such that the $k$th row has pivots at the $k$th element of \[ \ee_{i_1}, \cdots, \ee_{i_{|I|}}, \ff_{j_{|J|}}, \cdots \ff_{j_1},\]
with additional relations (equations \eqref{eq:isotropy1}, \eqref{eq:isotropy2}) imposed on the entries by the Lagrangian condition.

More precisely, an $n \times 2n$ matrix with complex entries, in the reduced row echelon form of corresponding to $J$ has its $k$th row vector as $[\vv_k]_{\{\ee, \ff\}}$ if $1 \le k \le |I|$ and as $[\ww_{n-k+1}]_{\{\ee, \ff\}}$ if $|I| < k \le n$, where
\begin{eqnarray*}
\vv_a &:=& \sum_{\substack{1 \le i < i_a\\ i \notin I}} r_{ai}\ee_i + \ee_{i_a} \quad 1 \le a \le |I|\\
&=:& \sum_{\substack {1 \le i \le n \\ i \notin I}} r_{ai} \ee_i\\
\ww_b &:=&\sum_{\substack{ 1 \le j \le n \\ j \notin I}} q_{bj} \ee_j + \sum_{\substack{j_b < \ell \le n\\ \ell \notin J}} p_{b\ell} \ff_{\ell} + \ff_{j_b} \quad 1 \le b \le |J|\\
&=:& \sum_{\substack{ 1 \le j \le n \\ j \notin I}} q_{bj} \ee_j + \sum_{\substack{1 \le \ell \le n\\ \ell \notin J}} p_{b\ell} \ff_{\ell}.
\end{eqnarray*}
The additional relations imposed by the isotropy condition are the following
\begin{eqnarray}
    \omega^\C(\vv_a, \ww_b) = &r_{aj_b} + p_{b i_a} & = 0 \quad \text{if } i_a > j_b \label{eq:isotropy1}\\
    \omega^\C(\ww_a, \ww_b) = &q_{aj_b} - q_{bj_a} & = 0.\label{eq:isotropy2}
\end{eqnarray}

We can assign and count the number of independent parameters is as follows. All the coefficients of $\vv_a$'s are determined by the coefficients of $\ww_b$'s, and among the $(2n - j_b) - (|I| + (|J| - b))$ possibly nonzero, nonpivot coefficients of $\ww_b$, $b-1$ are determined by the coefficients of $\ww_{b'}$ for $b' < b$. So each $\ww_b$ contributes $j_b^\vee:= n - j_b + 1$ additional independent parameters, and there are a total of $\sum_{b = 1}^{|J|} j_b^\vee$ independent parameters.

The sequence $j_1^\vee > \cdots > j_{|J|}^\vee$ (or equivalently, $J^\vee := \{ j_b^\vee\}_{b = 1}^{|J|}$) determines a shifted Young diagram
\[ \lambda_{J^\vee}:= \{ (k, \ell) \in \Z \times \Z : 1 \le k \le |J|,\quad k \le \ell \le j_k^\vee + k - 1\}\]
and we will identify $\lambda_{J^\vee}$ with the box diagram corresponding to it. Following the indexing conventions of entries of matrices, we will take the first component of $\Z \times \Z$ to denote the vertical position of the boxes, and the second component of $\Z \times \Z$ to denote the horizontal position of the boxes. The cardinality of $\lambda_{J^\vee}$ is $|\lambda_{J^\vee}|= \sum_{b = 1}^{|J|} j_b^\vee$. We will denote an element of $\{\lambda_{J^\vee}\}_{J \subseteq [n]}$ by $\lambda$ to refer to a shifted Young diagram without having to specify its shape. We will denote both the empty subset of $[n]$ and the empty shifted Young diagram by $\phi$, condoning the notation $\lambda_\phi = \phi$.

Let
\[
    z_{ k \ell} :=\begin{cases} q_{\ell j_k} & \text{if}\quad 1 \le k \le |J|, \quad k \le \ell \le |J|\\
    -r_{n+1 - \ell, j_k} & \text{if} \quad 1 \le k \le |J|, \quad |J| \le \ell \le n.
\end{cases}
\]

Suppressing the dependence on the fixed choice of Darboux basis, let $C_{\lambda_{J^\vee}}$ (or $C_\lambda$) be the set of complex Lagrangian subspaces of $\C^{2n}$ consisting of row spaces of $n\times 2n$ complex matrices in the reduced row echelon form corresponding to $J$ (or $\lambda$).

Denote by $\C_{(k, \ell)} := \C \times\{(k, \ell)\} \subseteq \C \times \Z^2$. Then we have homeomorphisms
\[ z^\lambda: C_\lambda \cong \prod_{(k, \ell) \in \lambda} \C_{(k, \ell)}.\]
When $\lambda = \phi$ we identify $\prod_{(k, \ell) \in \phi} \C_{(k, \ell)}$ with $\{0\}$.
We will assume the lexicographic order for the elements of $\lambda$, so that $(k, \ell) < (k', \ell')$ if either $k < k'$, or both $k = k'$ and $\ell < \ell'$. With this ordering we will denote again
\[ z^\lambda: C_\lambda \cong \C^{|\lambda|}.\]

From the existence and uniqueness of reduced row echelon forms, we have
\[ \Lag^\C(\R^{2n}) = \bigsqcup_{\lambda \subseteq \lambda_{[n]}} C_{\lambda}.\]

\begin{remark}[Arnol'd stratification]

If $\Gamma_{\phi}$ is the unique element of $C_\phi$, then $\dim_{\C} (\Gamma \cap \Gamma_{\phi}) = k$ if and only if $\Gamma \in C_{\lambda_{J^\vee}}$ and $|I| = k$ (cf 3.2.0 of Arnol'd \cite{Arnold} ). Suppose $\overline{j}_b^\vee:= j_{|J|+1-b}^\vee$, and $\pr_{\{\ff\}}$ denotes the projection from $\C^{2n}$ to 
\[ \span_{\C} \{\ff_1, \cdots, \ff_n\}\] along 
\[\span_{\C}(\ee_1, \cdots, \ee_n).\] Then $\Gamma \in C_{\lambda_{J^\vee}}$ if and only if
\[ \pr_{\{\ff\}} \Gamma \cap \span\{\ff_n, \cdots, \ff_{n+1-\overline{j}}\} =  b \quad \text{for} \quad \overline{j}_b^\vee \le  \overline{j} < \overline{j}_{b+1}^\vee.\]
For real coefficients, the linear extension of $\ee_i \mapsto \ff_{n+1-i}$, $\ff_i \mapsto -\ee_{n+1-i}$ maps $C_{\lambda_{J^\vee}}$ to $e\{\overline{j}_1^\vee, \cdots \overline{j}_{|J|}^\vee\}$ in 2.2 of Fuchs \cite{FuchsClassical}. 
\end{remark}

\begin{remark}[Maslov cycles and Maslov-Arnol'd cycles]
$C^{\R}_{\lambda_{[n]}}$ is the contractible open dense subset of all the Lagrangian subspaces transverse to the unique Lagrangian subspace of $C^{\R}_{\phi}$. So $\Lag^{\R}(\R^{2n}) \setminus C^{\R}_{\lambda_{[n]}} = \bigsqcup_{\lambda \subsetneq \lambda_{[n]}} C_{\lambda}^{\R}$ is the Maslov cycle (Sections 3.5, 3.6 of Arnold \cite{Arnold}, Section 2 of Robbin-Salamon \cite{RobbinSalamon}). To compare with the Maslov-Arnol'd homology cycles $A_k$, $B_k$, and $\Lambda_n^k$ of Fuks \cite{FuchsMA}, we note that $\dim_\C (\Gamma \cap \span_\C \{\ee_1, \cdots, \ee_k\}) = \ell$ if and only if $\Gamma \in C_{\lambda_{J^\vee}}$ where $|I \cap [k]| = \ell$.
\end{remark}

\begin{example}
If $J = \{1, 3, 4\}$ and $I =\{ 2, 5\}$, the pivot locations are at $\ee_2$, $\ee_5$, $\ff_4$, $\ff_3$, $\ff_1$, and the matrix is of the form
\[
\left[
\begin{array}{ccccc|ccccc}
\ast & 1 & 0 & 0 & 0 & 0 & 0 & 0 & 0 & 0 \\
\ast & 0 & \ast & \ast & 1 & 0 & 0 & 0 & 0 & 0 \\
\ast & 0 & \ast & \ast & 0 & \ast & 1 & 0 & 0 & 0\\
\ast & 0 & \ast & \ast & 0 & \ast & 0 & 1 & 0 & 0 \\
\ast & 0 & \ast & \ast & 0 & \ast & 0  & 0 & \ast & 1
\end{array}
\right].
\]

Since the pivot locations are at $r_{a i_a} = p_{b j_b} = 1$ (or $p_{a j_a} = p_{b j_b} = 1$), the position of the entries $r_{a j_b}, r_{a i_a}, p_{b j_b}, p_{b j_a}$ (or $q_{a j_b}, p_{a j_a}, p_{b j_b}, q_{b j_a}$) in the matrix are vertices of a parallelogram.

Equations \eqref{eq:isotropy1} are $r_{11} = -p_{12}$, $r_{21} = -p_{15}$, $r_{23} = -p_{25}$, and $r_{24} = -p_{35}$ so that the matrix form is

\[
\left[
\begin{array}{ccccc|ccccc}
-p_{12} & 1 & 0 & 0 & 0 & 0 & 0 & 0 & 0 & 0 \\
-p_{15} & 0 & -p_{25} & -p_{35} & 1 & 0 & 0 & 0 & 0 & 0 \\
\ast & 0 & \ast & \ast & 0 & p_{35} & 1 & 0 & 0 & 0\\
\ast & 0 & \ast & \ast & 0 & p_{25} & 0 & 1 & 0 & 0 \\
\ast & 0 & \ast & \ast & 0 & p_{15} & 0  & 0 & p_{12} & 1
\end{array}
\right]
\]
and equations \eqref{eq:isotropy2} are $q_{31} = q_{14}$, $q_{33}=q_{24}$, and $q_{21} = q_{13}$
\[
\left[
\begin{array}{ccccc|ccccc}
\ast & 1 & 0 & 0 & 0 & 0 & 0 & 0 & 0 & 0 \\
\ast & 0 & \ast & \ast & 1 & 0 & 0 & 0 & 0 & 0 \\
q_{14} & 0 & q_{24} & q_{34} & 0 & \ast & 1 & 0 & 0 & 0\\
q_{13} & 0 & q_{23} & q_{24} & 0 & \ast & 0 & 1 & 0 & 0 \\
q_{11} & 0 & q_{13} & q_{14} & 0 & \ast & 0 & 0 & \ast & 1
\end{array}
\right].
\]

$J^\vee = \{5, 3, 2\}$, so the corresponding shifted Young diagram is
\[ \lambda_{J^\vee} = 
\ytableausetup{centertableaux, boxsize = 5pt} \begin{ytableau} \none & \none & & &\none \\ \none & &&& \none\\ &&&&\end{ytableau}.\]

The independent parameters are labelled by
\[
\left[
\begin{array}{ccccc|ccccc}
-z_{15} & 1 & 0 & 0 & 0 & 0 & 0 & 0 & 0 & 0 \\
-z_{14} & 0 & -z_{24} & -z_{34} & 1 & 0 & 0 & 0 & 0 & 0 \\
z_{13} & 0 & z_{23} & z_{33} & 0 & \ast & 1 & 0 & 0 & 0\\
z_{12} & 0 & z_{22} & \ast & 0 & \ast & 0 & 1 & 0 & 0 \\
z_{11} & 0 & \ast & \ast & 0 & \ast & 0 & 0 & \ast & 1
\end{array}
\right],
\]
where, as in Cartesian coordinates, the horizontal location is given by the first index. Equivalently, the independent parameters are labelled by
\[
\left[
\begin{array}{ccccc|ccccc}
\ast & 1 & 0 & 0 & 0 & 0 & 0 & 0 & 0 & 0 \\
\ast & 0 & \ast & \ast & 1 & 0 & 0 & 0 & 0 & 0 \\
\ast & 0 & \ast & z_{33} & 0 & z_{34} & 1 & 0 & 0 & 0\\
\ast & 0 & z_{22} & z_{23} & 0 & z_{24} & 0 & 1 & 0 & 0 \\
z_{11} & 0 & z_{12} & z_{13} & 0 & z_{14} & 0 & 0 & z_{15} & 1
\end{array}
\right],
\]
where, as in matrix component notation, the horizontal location is given by the second index. Asterisks denote dependent parameters. The labels for the independent parameters can be read off from their corresponding locations in the shifted Young diagrams (Figure \ref{fig:labelling}).
\begin{figure}
\centering \includegraphics[width = 0.2\textwidth]{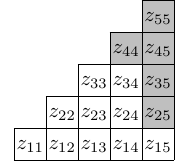}
\caption{The labels for the independent parameters can be read off from their corresponding locations in the shifted Young diagrams.}
\label{fig:labelling}
\end{figure}

\end{example}

\subsection{Attaching maps of Schubert cells}

We will denote by $\lambda \le \lambda'$ if $\lambda$ is contained in $\lambda'$, or equivalently, if there exist $J, J' \subseteq [n]$ such that $\lambda = \lambda_{J^\vee}$, $\lambda' = \lambda_{(J')^\vee}$, $|J| \le |J'|$ and $j_k^\vee \le (j_k')^\vee$ for all $1 \le k \le |J|$. Since every nonempty $\lambda$ contains $(1, 1)$, the containment of shifted Young diagrams assumes they are aligned on the bottom left corner. For a comparison with the Bruhat ordering, we refer the reader to Proposition 6.1 of Ikeda and Naruse \cite{IkedaNaruse}.

Every $C_\lambda$ is homeomorphic to $\C^{|\lambda|}$, so when tubular neighborhoods of $C_\lambda$ inside some $\overline{C_{\lambda'}}$ exist, they are homeomorphic to trivial bundles over $C_\lambda$. We will construct such homeomorphisms in detail when $\lambda$, $\lambda'$ are such that $\lambda \le \lambda'$ and $|\lambda'| = |\lambda|+1$ to obtain information about degrees of attaching maps.

There are two cases to consider, i) when $\lambda'$ has more rows than $\lambda$, and ii) when $\lambda$ and $\lambda'$ have the same number of rows. In both cases we setup some notation, that $\lambda = \lambda_{J^\vee}$ and $\lambda' = \lambda_{(J')^\vee}$, and that $\lambda' \setminus \lambda = \{ (\star, \star + (j_\star')^\vee - 1)\}$.

If $|J'| > |J|$ then $\star = |J'| = |J|+1$, and $(j_{\star}')^\vee = 1$. So $j_{\star}' = n$.  Since all $j_k^\vee = (j_k')^\vee$ for $1 \le k \le |J|$, $j_\star' \notin J$, so  $j_\star' \in I$, and $I' = I \setminus\{j_{\star}'\}$, $J' = J \sqcup \{j_{\star}'\}$, and $i_{|I|} = j_{\star}' = n$. Suppose the reduced row echelon form of $\Gamma \in C_\lambda$ has row vectors $\{\vv_a, \ww_b\}_{1 \le a \le |I|, 1 \le b \le |J|}$. Let
\[\ww'_t := \vv_{|I|} + t \ff_{j_\star'} \quad t \in \C. \]
Then $\Gamma_t := \span_\C \{ \vv_a, \ww'_t, \ww_b\}_{1 \le a \le |I|-1, 1 \le b \le |J|}$ is always complex Lagrangian, because
\begin{eqnarray}
 \omega^\C(\ww_t', \vv_a) &=& 0 \nonumber\\
    \omega^\C(\ww_t', \ww_b) &=& -tq_{b, j_\star'} = 0.\label{eq:newLagrangian}
\end{eqnarray}
When $t\neq 0$, we can row reduce these vectors as
\begin{eqnarray}
\vv'_a(t)  &:=& \vv_a \quad \quad \quad \quad \quad \quad i_a < j_\star' = n \nonumber\\
\ww'_{\star}(t) &:=& t^{-1} \ww'_t \nonumber \\
\ww'_b (t) &:=& \ww_b - t^{-1} z_{b, \star}\ww_t' \quad j_b < j_\star' = n. \label{eq:rrednew}
\end{eqnarray}

So $\Gamma_t$ is represented by the $n \times 2n$ complex matrix with $k$th row $[\vv'_k(t)]_{\{ \ee, \ff\}}$ for $1 \le k \le |I|-1$, $[\ww_\star'(t)]_{\{\ee, \ff\}}$ for $k = |I|$, and $[\ww'_{n-k+1}(t)]_{\{\ee, \ff\}}$ for $1 \le k \le |J|$. It is in reduced row echelon form with pivots at
\[\ee_{i_1}, \cdots, \ee_{i_{|I|-1}}, \ff_{j_\star'}, \ff_{j_{|J|}}, \cdots \ff_{j_1}.\]
So $\Gamma_t \in C_{\lambda_{(J')^\vee}}$ whenever $t \in \C^\times$. By construction $\Gamma_0 = \Gamma$, so
\[ \varphi_{\lambda}^{\lambda'} (\Gamma, t) := \Gamma_t\]
is a local homeomorphism from $C_\lambda \times \C^{\times}$ to $C_{\lambda'}$.

If $|J| = |J'|$, then $(j_k')^\vee = j_k'$ for all $1 \le k \le |J|$, except when $k = \star$, in which case $(j_\star')^\vee = j_\star^\vee + 1$, so $j_\star' + 1= j_\star$. This also implies that $(j_\star')^\vee \notin J^\vee$ and $j_\star^\vee \notin (J')^\vee$. So $I' = I \sqcup \{ j_\star\} \setminus \{j_\star'\}$ and $J' = J \sqcup \{j_\star'\} \setminus \{ j_\star\}$. Suppose the reduced row echelon form of $\Gamma \in C_\lambda$ has row vectors $\{\vv_a, \ww_b\}_{1 \le a \le |I|, 1 \le b \le |J|}$. Let $\circledast$ be such that $i_\circledast = j_{\star}'$, and
\begin{eqnarray*}
    \vv'_t &:=& \vv_{\circledast} - t \ee_{j_\star}\\
    \ww'_t &:=& \ww_{\star} + t \ff_{j_\star'} \quad t\in \C.
    \end{eqnarray*}
Then $\Gamma_t := \span_\C \{ \vv_a, \vv'_t, \ww'_t, \ww_b\}_{i_a \in I \setminus \{ j_\star '\}, j_b \in J \setminus \{ j_\star\}}$ is complex Lagrangian for all $t \in \C$, because
\begin{eqnarray}
    \omega^\C(\vv_t', \ww_b) =& -t\cdot p_{b, j_\star} &= 0 \nonumber\\
    \omega^\C(\ww_t', \ww_b) =& t \cdot q_{b, j_\star'}& = 0 \nonumber\\
    \omega^\C(\ww_t', \vv_a) =& t \cdot r_{a, j_\star'}& = 0 \nonumber\\
    \omega^\C(\vv_t', \ww_t') =& -t + t &= 0. \label{eq:sameLagrangian}
\end{eqnarray}
When $t \neq 0$, we can row reduce these vectors as follows
\begin{equation} \label{eq:rredsamev}
\vv'_a (t) := \begin{cases} \vv_a & \text{ if } i_a < j_\star'\\
-t^{-1}\vv_t' & \text{ if } i_a = j_\star' \\
\vv_a - t^{-1} z_{\star, n+1-a} \vv_t' & \text{ if } i_a > j_\star'.\end{cases}
\end{equation}
and
\begin{equation} \label{eq:rredsamew}
\ww'_b (t) := \begin{cases} \ww_b + t^{-1} z_{\star, b} \vv_t' & \text{ if } j_b > j_\star \\
t^{-1} \ww_t' + t^{-2} z_{\star, \star}\vv_t' & \text{ if } j_b = j_\star \\
\ww_b + t^{-1} z_{b, \star} \vv_t' - z_{b, \star + j_\star^\vee}\ww_\star'(t) & \text{ if } j_b < j_\star.\end{cases}
\end{equation}
Then the $n\times 2n$ complex matrix with $k$th row $[\vv'_k(t)]_{\{ \ee, \ff\}}$ for $1 \le a \le |I|$, $[\ww'_{n-k+1}(t)]_{\{\ee, \ff\}}$ for $1 \le k \le |J|$ is in reduced row echelon form with pivots at \[\ee_{i_1}, \cdots, \ee_{j_\star}, \cdots, \ee_{i_{|I|}}, \ff_{j_{|J|}}, \cdots, \ff_{j_{\star - 1}}, \ff_{j_\star'}, \ff_{j_{\star + 1}}, \cdots, \ff_{j_1}.\] Thus $\Gamma_t \in C_{\lambda'}$ when $t \in \C^\times$.

By construction $\Gamma_0 = \Gamma$, and again $\varphi_{\lambda}^{\lambda'} (\Gamma, t) := \Gamma_t$  is a local homeomorphism from $C_\lambda \times \C^{\times}$ to $C_{\lambda'}$.

We can obtain charts of $\overline{C_{\lambda'}}$ in the neighborhood of points in $C_\lambda$ from local inverses of $\varphi_{\lambda}^{\lambda'}$. From the row reduction prescription of equations \eqref{eq:rrednew}, \eqref{eq:rredsamev}, and \eqref{eq:rredsamew} we can explicitly compute the transition map of $C_{\lambda'}$
\[z^{\lambda'} \circ \varphi_{\lambda}^{\lambda'} \circ ((z^{\lambda})^{-1} \times \Id_{\C^\times}): \C^{|\lambda|} \times \C^\times \to \C^{|\lambda'|}.\]

If $|J'| > |J|$, they are

\begin{subnumcases}{\label{eq:znew}  z_{k\ell}' =}
z_{k\ell} + t^{-1} z_{k, \star} z_{\ell, \star} & \text{ if } $\ell < \star$ \label{eq:leftofcolumn}\\
-t^{-1}z_{k\ell} & \text{ if } $\ell < k =\star$ \label{eq:undercolumn}\\
z_{k\ell} & \text{ if } $\ell > \star$ \label{eq:rightofcolumn}\\
t^{-1} & \text{ if } $k = \ell = \star.$ \label{eq:topcolumn}
\end{subnumcases}

If $|J'| = |J|$, they are
\begin{subnumcases}{\label{eq:zsame} z_{k\ell}' =}
z_{k\ell}  - t^{-1} z_{k, \star} z_{\ell, \star + j_\star^\vee} -t^{-1} z_{\ell, \star} z_{k, \star + j_\star^\vee} \notag \\ \quad + t^{-2} z_{\star, \star} z_{k, \star + j_\star^\vee} z_{\ell, \star + j_\star^\vee} & if $k \le \ell < \star$ \label{eq:leftofarch}\\
    z_{k\ell} + t^{-1} z_{\star, \ell} z_{k, \star + j_\star^\vee} & \text{ if } $k < \star < \ell < \star + j_\star^\vee$ \label{eq:underarch}\\
    t^{-1} z_{k\ell} - t^{-2} z_{\star, \star} z_{k, \star + j_\star^\vee} & \text{ if } $k < \star = \ell$\label{eq:leftarch}\\
    t^{-1} z_{k\ell} & \text{ if } $k = \star < \ell < \star + j_\star^\vee$\label{eq:toparch}\\
    -t^{-1} z_{k\ell} & \text{ if } $k < \star, \ell = \star + j_\star^\vee$\label{eq:rightarch}\\
    t^{-2} z_{k\ell} & \text{ if } $k = \ell = \star$ \label{eq:leftcornerarch}\\
    t^{-1} & \text{ if } $k = \star, \ell = \star + j_\star^\vee$\label{eq:rightcornerarch}\\
    z_{k\ell} & \text{ if } $k > \star \text{ or } \ell > \star+ j_\star^\vee$ \label{eq:topandrightofarch}

\end{subnumcases}
for $(k, \ell) \in \lambda'$. We refer to Figure \ref{fig:zsame}, \ref{fig:ACR}, and Remark \ref{rmk:ACR} for a diagrammatic description of the nine cases in terms of shifted Young diagrams.

\begin{figure}
\centering \includegraphics[width = 0.6\textwidth]{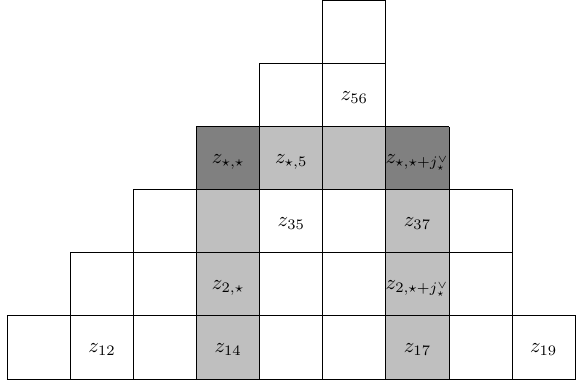}
\caption{Terms appearing in some equations of \eqref{eq:zsame}: $z_{12}' = z_{12}- t^{-1} z_{2, \star} z_{1, \star + j_\star^\vee} -t^{-1} z_{1, \star} z_{2, \star + j_\star^\vee} + t^{-2} z_{\star, \star} z_{1, \star + j_\star^\vee} z_{2, \star + j_\star^\vee}$, $z_{35}' = z_{35} - t^{-1} z_{\star, 5} z_{37}$, $z_{56}' = z_{56}$, $z_{19}' = z_{19}$. Additional terms are in terms of variables succeeding in the lexicographical order.}
\label{fig:zsame}
\end{figure}

Because of the shape of the row reduced echelon form, and the fact that we only subtract rows above from rows below, additional terms are only dependent on variables succeeding in the lexicographical order. So the complex Jacobian determinant of the transition map is upper triangular.

We can compute the diagonal entries. When $|J'| > |J|$, they are
\begin{eqnarray*}
    \frac{\partial z_{k\ell}'}{\partial z_{k\ell}} &=& \begin{cases} 1 & \text{ if } (k, \ell) \in \lambda_{J^\vee}, \ell \neq \star\\
    -t^{-1} &\text{ if } (k, \ell) \in \lambda_{J^\vee}, \ell = \star\end{cases}\\
    \frac{\partial z_{\star\star}'}{\partial t} &=& -t^{-2}.
\end{eqnarray*}

When $|J'| = |J|$, they are
\[\frac{\partial z_{k\ell}'}{\partial z_{k\ell}} = \begin{cases} 1 & \text{ if } k \neq \star, \ell \neq \star, \star + j_\star^\vee\\
    t^{-1} &\text{ if } k \neq \star, \ell = \star \text { or } k = \star, \ell \neq \star, \star + j_{\star}^\vee\\
    -t^{-1}&\text{ if } k \neq \star, \ell = \star + j_\star^\vee\\
    t^{-2} &\text{ if } k = \ell = \star
    \end{cases}\]
for $(k, \ell) \in \lambda_{J^\vee}$ and
\[ \frac{\partial z_{\star, \star + j_\star^\vee}'}{\partial t} = -t^{-2}.\]

Finally, since the new variable $z_{\star, \star + j_\star^\vee}'= t^{-1}$ its position in the lexicographic order does not affect the upper triangularity of the complex Jacobian matrix, but contributes an overall sign of $(-1)^{\sum_{k = \star + 1}^{|J|} j_k^\vee}$ to the complex determinant.

So the complex Jacobian determinant of the transition map is equal to
\begin{equation}\label{eq:Jacobian} \begin{cases} (-1)^{\star} \cdot t^{-\star-1} & \text{ if }  |J'| > |J| \\ (-1)^{\star + \sum_{k = \star + 1}^{|J'|} j_k^\vee} \cdot t^{-2\star - j_{\star}^\vee -1} & \text{ if } |J'| = |J|. \end{cases}
\end{equation}
We will denote this function as $T_{\lambda}^{\lambda'}(z^\lambda, t) = T_{\lambda}^{\lambda'}(t)$.

\begin{example}[$|J'| > |J|$ case]
Suppose $I = \{3\}$, $J = \{1, 2\}$, $I' = \phi$, $J' = \{1, 2, 3\}$.
Then $\Gamma_t$ is the complex row space of some
\[ \left[
\begin{array}{ccc|ccc}
-z_{13} & -z_{23} & 1 & t & 0 & 0\\ z_{12} & z_{22}& 0 & z_{23} & 1 & 0\\
z_{11} & z_{12} & 0 & z_{13} & 0 & 1
\end{array}\right].\]
This matrix row reduces to
\[ \left[
\begin{array}{ccc|ccc}
-t^{-1} z_{13} & - t^{-1} z_{23} & t^{-1} & 1 & 0 & 0\\ z_{12} + t^{-1}  z_{13} z_{23} & z_{22} + t^{-1} z_{23}^2 & -t^{-1}z_{23}& 0 & 1 & 0\\
z_{11} + t^{-1}  z_{13}^2 & z_{12} + t^{-1}  z_{13} z_{23} & -t^{-1} 
 z_{13}& 0 & 0 & 1
\end{array}\right].\]

The complex Jacobian matrix of $z^{\lambda'} \circ \varphi_{\lambda} ^{\lambda'} \circ ((z^{\lambda})^{-1} \times \Id_{\C^\times})$ is
\[ \frac{\partial(z_{11}', z_{12}', z_{13}', z_{22}', z_{23}', z_{33}')}{\partial(z_{11}, z_{12}, z_{13}, z_{22}, z_{23}, t)} = \begin{pmatrix} 1 & 0 & 2t^{-1}  z_{13} & 0 & 0 & -t^{-2}  z_{13}^2\\
0 & 1 & t^{-1}  z_{23} & 0 & t^{-1}  z_{13} & -t^{-2} z_{13}z_{23}\\
0 & 0 & -t^{-1} & 0 & 0 & t^{-2}  z_{13} \\
0 & 0 & 0 & 1 & 2t^{-1}  z_{23} & -t^{-2}  z_{23}^2\\
0 & 0 & 0 & 0 & -t^{-1} & t^{-2}  z_{23}\\
0 & 0 & 0 & 0 & 0 & -t^{-2} \end{pmatrix}.\]
This matrix has complex determinant $(-1)^3 \cdot t^{-4}$. The diagonal entries can be schematically represented as \ref{fig:DiagonalEntries}. 

\begin{figure}
\centering
\includegraphics[width = 0.2\linewidth]{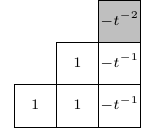} \quad \includegraphics[width = 0.3\linewidth]{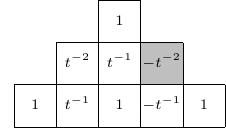}
\caption{Diagonal entries of Jacobians of some transition maps}
\label{fig:DiagonalEntries}
\end{figure}
\end{example}

\begin{example}[$|J'| = |J|$ case]
Suppose $I = \{2, 3\}$, $J = \{1, 4, 5\}$, $I' = \{2, 4\}$, and $J' = \{1, 3, 5\}$. Then $\Gamma_t$ is the complex row space of some
\[ \left[
\begin{array}{ccccc|ccccc}
-z_{15} & 1 & 0 & 0 & 0 & 0 & 0 & 0 & 0& 0\\
-z_{14} & 0 & 1 & -t & 0 & 0 & 0 & 0 & 0 & 0\\
z_{13} & 0 & 0 & z_{23} & z_{33} & 1 & 0 & 0 & 0 & 0 \\
z_{12} & 0 & 0 & z_{22} & z_{23} & 0 & 1 & t & 0 & 0 \\
z_{11} & 0 & 0 & z_{12} & z_{13} & 0 & 0 & z_{14} & z_{15} & 1
\end{array}\right].\]
After row reduction, the transition map $z^{\lambda'} \circ \varphi_{\lambda} ^{\lambda'} \circ ((z^{\lambda})^{-1} \times \Id_{\C^\times})$ is the following
\begin{eqnarray*}
    z_{11}' &=& z_{11} - 2t^{-1}  z_{12} z_{14}  + t^{-2}  z_{22} z_{14}^2\\
    z_{12}' &=& t^{-1}  z_{12} - t^{-2}  z_{22} z_{14}\\
    z_{13}' &=& z_{13} - t^{-1} z_{23} z_{14}\\
    z_{14}' &=& -t^{-1}  z_{14}\\
    z_{15}' &=& z_{15}\\
    z_{22}' &=& t^{-2}  z_{22}\\
    z_{23}' &=& t^{-1}  z_{23}\\
    z_{33}' &=& z_{33}\\
    z_{24}' &=& t^{-1}
    \end{eqnarray*}
The complex Jacobian matrix is upper triangular in the lexicographic ordering, and the diagonal entries can be represented schematically in Figure \ref{fig:DiagonalEntries}. The product is $(-1)^{2+1} \cdot t^{-7}$.
\end{example}

\begin{remark}[Arches, columns, and roofs]\label{rmk:ACR}
When $\lambda' \setminus \lambda = \{ (\star, \star + (j'_\star)^{\vee}-1)\}$, let
\begin{eqnarray*}
    \alpha_{\star} &:=& \{(k, \star), (k, \star + (j_\star')^\vee-1): 1 \le k \le \star\} \cup \{ (\star, \ell) : \star \le \ell \le \star + (j_\star')^\vee-1\}.\\
    c_\star &:=& \{ (k, \star + (j_\star')^\vee-1): 1 \le k \le \star\}.\\
    \rho_\star &:=& \{ (k, \ell) \in \lambda : k > \star\}.
\end{eqnarray*}
In this notation, the complex Jacobian determinants of equation \eqref{eq:Jacobian} are
\begin{equation} \label{eq:JacobianACR} T_{\lambda}^{\lambda'}(t) = \begin{cases} (-1)^{|c_\star|} \cdot t^{-|c_\star|-1} & \text{ if }  |J'| > |J| \\ (-1)^{|c_\star| + |\rho_\star|} \cdot t^{-|\alpha_\star| - 2} & \text{ if } |J'| = |J| .\end{cases}
\end{equation}
\end{remark}

\begin{figure}
\centering
\includegraphics{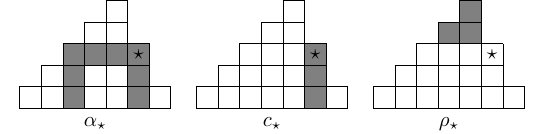}
\caption{$\alpha_\star$, $c_\star$, and $\rho_\star$ are shaded for $\lambda_{\{ 7, 5, 3, 2, 1\}} \le \lambda_{\{ 7, 5, 4, 2, 1\}}$}
\label{fig:ACR}
\end{figure}

\begin{proposition}[\label{thm:frontier}Frontier condition]
The following are equivalent
\begin{enumerate}
    \item $C_\lambda \cap \overline{C_{\lambda'}} \neq \phi$.
    \item $C_\lambda \subseteq \overline{C_{\lambda'}}$.
    \item $\lambda \le \lambda'$.
\end{enumerate}
\end{proposition}

\begin{proof}
$(3) \implies (2)$: There exists an increasing sequence $\lambda_1 \le \lambda_2 \le \cdots \lambda_\ell$ such that $\lambda_1 = \lambda$, $\lambda_\ell = \lambda'$, and $|\lambda_{k+1}| = |\lambda_k|+1$ for $1 \le k < \ell = |\lambda'| - |\lambda|$, and a sequence of embeddings $\{\varphi_{\lambda_k}^{\lambda_{k+1}}\}_{k = 1}^\ell$. For each $\Gamma \in C_\lambda$, let
\[ \Gamma_{t_1, \cdots, t_\ell} := \varphi_{\lambda_{\ell - 1}}^{\lambda_\ell}(\cdots(\varphi_{\lambda_2}^{\lambda_3} (\varphi_{\lambda_1}^{\lambda_2} (\Gamma, t_1), t_2), \cdots), t_\ell) \quad t_1, \cdots,  t_\ell \in \C\]
where if $t_j = 0$ for some $j \in \{1, \cdots, \ell\}$, then $t_k = 0$ for all $k \ge j$.
If $t_1, \cdots, t_\ell$ are all nonzero, then $\Gamma_{t_1, \cdots, t_\ell} \in C_{\lambda'}$ and $\Gamma_{t_1, \cdots, t_\ell} \to \Gamma$ as $(t_1, \cdots, t_\ell) \to 0$.
$(2) \implies (1)$ is immediate, as $C_\lambda$ is nonempty. $(1) \implies (3)$: Suppose $\lambda = \lambda_{J^\vee}$ and $\lambda' = \lambda_{(J')^\vee}$. If $\lambda \nleq \lambda'$, then there exists a $1 \le b \le |J|$ such that $j_b^\vee > (j_b')^\vee$. Then $p_{b j_b} (\Gamma') = 0$ for all $\Gamma' \in C_{\lambda'}$ and $p_{b j_b}(\Gamma) = 1$ for all $\Gamma \in C_{\lambda}$. So $C_\lambda \cap \overline{C_{\lambda'}} = \phi$.\end{proof}

We construct the attaching maps following the proof of Theorem 3.2.3 of Tajakka's thesis \cite{Tajakka} (cf Proposition 1.17 of Hatcher's book \cite{Hatcher}) for ordinary Grassmannians. To do this we set up some notation.

Let $\mathfrak{J}$ be the compatible linear complex structure on $\R^{2n}$ given by the linear extension of $\ee_i \mapsto \ff_i$, $\ff_i \mapsto -\ee_i$ for $1 \le i \le n$. Denote again by $\mathfrak{J}$ its $\C$-linear extension to $\C^{2n}$. Take the hermitian inner product $\langle \vv, \ww\rangle:= \omega^\C(\vv, \mathfrak{J}\overline{\ww})$ of $\C^{2n}$. Let $\Sp(n)$ be the group of complex linear transformations on $\C^{2n}$ that preserves this hermitian inner product and $\omega^\C$. Given a unitary basis $\{\uu_1, \cdots, \uu_n\}$ of a complex Lagrangian subspace $\Gamma$, $\{\mathfrak{J} \overline{\uu}_1, \cdots, \mathfrak{J}\overline{\uu}_n\}$ is a unitary basis of the complex Lagrangian subspace $\mathfrak{J}\overline{\Gamma}$. $\Gamma$ and $\mathfrak{J}\overline{\Gamma}$ are always orthogonal, so $\{ \uu_1, \cdots, \uu_n, \mathfrak{J}\overline{\uu}_1, \cdots, \mathfrak{J} \overline{\uu}_n\}$ is a unitary Darboux basis of $\C^{2n}$. Sending $\{\ee_1, \cdots, \ee_n, \ff_1, \cdots, \ff_n\}$ to this basis determines a unique element of $\Sp(n)$.

Let $\mathcal{F}_n$ be the set of $n$-tuples $\mathfrak{u}:=(\uu_1, \cdots, \uu_n)$ such that $\{\uu_1, \cdots, \uu_n\}$ is a unitary basis of some complex Lagrangian subspace of $\C^{2n}$. Taking each such $n$-tuple to the spanning space of its elements gives a principal $U(n)$-bundle $\pi: \mathcal{F}_n \to \Lag^\C(\R^{2n})$.

For each shifted Young diagram $\lambda = \lambda_{J^\vee}$, let $E_\lambda$ be the set of $(\uu_1, \cdots, \uu_n)$ in $\mathcal{F}_n$ such that
\[ \uu_k \in \begin{cases} \span_\C\{ \ee_1, \cdots, \ee_{i_k}\} & \text{ if } 1 \le k \le |I|\\ \span_\C\{ \ee_1, \cdots, \ee_n, \ff_n, \cdots, \ff_{j_{n+1-k}}\} & \text{ if } |I| < k \le n. \end{cases}\] Let $E(\lambda) \subseteq E_\lambda$ be the elements such that all pivot coefficients (i.e. $q_{i_k}(\uu_k)$ for $1 \le k \le |I|$ and $p_{j_{n+1-k}}(\uu_{k})$ for $|I| < k \le n$) are nonnegative real valued. If $\lambda \le \lambda'$ then $E_\lambda \subseteq E_{\lambda'}$ and $E(\lambda) \subseteq E(\lambda')$.

We can apply the Gram-Schmidt algorithm on the row vectors of the reduced row echelon forms of $\Gamma$'s in $C_{\lambda}$ such that pivot coefficients take values in nonnegative real numbers. Then we obtain unique \emph{unitary row echelon forms} corresponding to each $\Gamma \in C_\lambda$. We obtain continuous sections of $\pi: \pi^{-1} (\overline{C}_{\lambda}) \to \overline{C}_\lambda$.

We will refer to the set $\{ (\xi_1, \cdots, \xi_{k}) \in \C^{k} : \sum_{\ell = 1}^k |\xi_\ell|^2 = 1, \xi_k \in [0, 1]\}$ as the \emph{closed upper hemisphere of $S^{2k-1}$}.

\begin{proposition}[Construction of attaching maps] \label{prop:attachingmaps}
$E(\lambda)$ is homeomorphic to a closed $2|\lambda|$-ball. Consequently, the precomposition of this homeomorphism with the restriction of  $\pi$ to $E(\lambda) \setminus \pi^{-1}(C_\lambda)$ is the attaching map of $C_\lambda$ in the CW decomposition of $\Lag^\C(\R^{2n})$.
\end{proposition}

\begin{proof} Suppose $\lambda = \lambda_{J^\vee}$ for some $J \subseteq [n]$. We will prove by induction on $|J|$. If $J = \{k\}$, $E(\lambda_{\{k\}})$ consists of $(\ee_1, \cdots, \ee_{k-1}, \ee_{k+1}, \cdots, \ee_{n-1}, \uu_n)$, where $\uu_n$ takes values in the closed upper hemisphere of
\[ S^{2k} \subseteq \span_\C\{\ee_k, \ff_n, \ff_{n-1}, \cdots, \ff_{n-k+2}\} \oplus \span_\R\{\ff_{n-k+1}\}.\] So
$E(\lambda_{\{k\}})$ is homeomorphic to a closed $2|\lambda|$-ball.

If $|J| > 1$, let $\pr((\uu_1, \cdots, \uu_n)) := \uu_n$. Then $\mathfrak{u} \in \pr^{-1}(\ff_{j_1})$ can be represented by $n \times 2n$ matrix with $k$th row equal to $[\uu_k]_{\{\ee, \ff\}}$ has only zeros on the $j_1$st and $2n-j_1 + 1$st columns and zeros on the $n$th row except at the `pivot' location $\ff_{j_1}$. So $\pr^{-1}(\ff_{j_1})$ is homeomorphic to $E(\lambda_{\{ j_2^\vee, \cdots, j_{n}^\vee\}}) \subseteq \mathcal{F}_{n-1}$.

For $\mathfrak{u} = (\uu_1, \cdots, \uu_n) \in E(\lambda)$, obtain $(\hat{\uu}_1, \cdots, \hat{\uu}_{n-1})$ by applying Gram-Schmidt to $\{ \uu_k - \omega^\C(\uu_k, \ff_{j_1}) \ee_{j_1}\}_{1 \le k \le n-1}$, and let $\hat{\uu}_n := \ff_{j_1}$. Then $\{\hat{\uu}_1, \cdots, \hat{\uu}_n\}$ is a unitary basis of a complex Lagrangian subspace. Let $S_{\mathfrak{u}}$ be the $\C$-linear extension of $\uu_k \mapsto \hat{\uu}_k$, $\mathfrak{J} \overline{\uu}_k \mapsto \mathfrak{J} \overline{\hat{\uu}}_k$  for $1 \le k \le n$. Then $S_\mathfrak{u} \in \Sp(n)$.

If $i_k < j_1$, then $\omega^\C(\uu_k, \ff_{j_1}) = 0$, so sending $\mathfrak{u}$ to $\hat{\mathfrak{u}}$ is equivalent to multiplying the $n \times 2n$ matrix with $k$th row as $[\uu_k]_{\{\ee, \ff\}}$, on the left by a lower triangular $n \times n$ matrix. So $(\hat{\uu}_1, \cdots, \hat{\uu}_n) \in \pr^{-1}(\ff_{j_1}) \subseteq E(\lambda)$. Assigning $S_\mathfrak{u}$ to $\mathfrak{u}$ also depends continuously on $\mathfrak{u}$. Let $E(\lambda)'$ be the image $(\pr, S_\cdot)(E(\lambda)) \subseteq S^{2(2n - j_1 +1) -1} \times \pr^{-1}(\ff_{j_1})$. Consider the projection onto the second factor. The fiber above $(\hat{\uu}_1, \cdots, \hat{\uu}_{n-1}, \ff_{j_1})$ consists of all $\uu_n$ in the intersection of the closed upper hemisphere of $S^{2(2n-j_1 +1)-1}$ with the orthogonal complement of
\[ \span_{\R} \{ \uu_1, \mathfrak{J} \overline{\uu}_1, \cdots, \uu_{n-1}, \mathfrak{J} \overline{\uu}_{n-1}, \ff_{1}, i \ff_{1} \cdots, \ff_{j_1-1}, i \ff_{j_1-1}, i \ff_{j_1}\}\]
which is a closed hemisphere of real dimension $(4n - (2(n-1) + 2j_1 + 1))-1 = 2 j_1$.
So $E(\lambda) \cong E(\lambda)' \to \pr^{-1}(\ff_{j_1}) \cong E(\lambda_{\{ j_2^\vee, \cdots, j_{|J|}^\vee\}})$ is a trivial bundle with fibers homeomorphic to the closed $2j_1$-ball. By inductive hypothesis, the base is homeomorphic to a closed $2|\lambda| - 2j_1$-ball. The total space $E(\lambda)$ is thus homeomorphic to a closed $2|\lambda|$-ball.
\end{proof}

\begin{remark}[Integral homology of real and complex Lagrangian Grassmannians]

From Proposition \ref{prop:attachingmaps},
\[ \Lag^\C(\R^{2n}) = \bigsqcup_{\lambda} C_\lambda\]
is a CW decomposition. Since all Schubert cells are even dimensional, all attaching maps have degree $0$, and the integral homology groups can computed by counting how many shifted Young diagrams of a particular size are allowed:
\begin{eqnarray*}
H_{2k} (\Lag^\C(\R^{2n}); \Z) &\cong& \Z^{|\{ \lambda : |\lambda| = k\}|}\\
H_{2k + 1} (\Lag^\C(\R^{2n}); \Z) &\cong& \{0\}.
\end{eqnarray*}

Equations \eqref{eq:Jacobian} and Propositions \ref{thm:frontier}, \ref{prop:attachingmaps}  also hold when all parameters are real numbers.
\[ \Lag^\R(\R^{2n}) = \bigsqcup_{\lambda} C_\lambda^\R\]
is also a CW decomposition, and the degree of the attaching map $g_{\lambda}^{\lambda'}: \partial C_{\lambda'}^\R \to C_{\lambda}^\R$ can be computed. By Proposition \ref{thm:frontier} $\deg g_{\lambda}^{\lambda'} = 0$ if $\lambda \nleq \lambda'$, and for dimensional reasons $\deg g_{\lambda}^{\lambda'} = 0$ if $\lambda \le \lambda'$ and $|\lambda'| > |\lambda| + 1$. When $\lambda_{J^\vee} \le \lambda_{(J')^\vee}$ and $|\lambda_{(J')^\vee}| = |\lambda_{J^\vee}| + 1$, then from equation \eqref{eq:Jacobian} we have

\begin{subnumcases}{\label{eq:bdyreal} \deg g_{\lambda_{J^\vee}}^{\lambda_{(J')^\vee}} =}(-1)^\star \cdot ( 1- (-1)^{\star +1}) & \text{ if } $|J'| > |J|$ \label{eq:bdyrealnew}\\
(-1)^{\star + \sum_{k = \star + 1}^{|J|} (j_k')^\vee} \cdot (1-(-1)^{(j_\star')^\vee}) & \text{ if } $|J'|= |J|$\label{eq:bdyrealsame} \end{subnumcases}
since $C_{\lambda_{J^\vee}}^\R \times \{t\}$ and $C_{\lambda_{J^\vee}}^\R \times \{-t\}$ have opposite orientations. This agrees with \cite{FuchsClassical} and \cite{Rabelo} up to sign. Integral homology can be computed algorithmically.
\end{remark}

\begin{example}
The eight Schubert cells of $\Lag^{\C}(\R^6)$ are shown in Table \ref{tab:LagCR6}.
\begin{table}
\ytableausetup{centertableaux, boxsize = 5pt} 
\begin{tabular}{c c c c c}\hline
\noalign{\smallskip}$J$ & $I$ & RREF & Diagram & $J^\vee$ \\ \hline
\noalign{\smallskip} $\{1, 2, 3\}$ & $\phi$ & $\left[
\begin{array}{ccc|ccc}
\ast & \ast &z_{33}& 1& 0 & 0\\ \ast &z_{22}& z_{23}& 0 & 1 & 0\\
z_{11} & z_{12} & z_{13} & 0 & 0 & 1
\end{array}\right]$ & \ydiagram{2 + 1, 1+2, 3} &$\{3, 2, 1\}$\\[15pt]

$\{1, 2\}$ & $\{3\}$ & $\left[
\begin{array}{ccc|ccc}
\ast & \ast & 1 & 0 & 0 & 0\\ \ast & z_{22}& 0 & z_{23} & 1 & 0\\
z_{11} & z_{12} & 0 & z_{13} & 0 & 1
\end{array}\right]$ & \ydiagram{1+2, 3} &$\{3, 2\}$\\[15pt]

$\{1, 3\}$ & $\{2\}$ & $\left[
\begin{array}{ccc|ccc}
\ast & 1 & 0 & 0& 0 & 0\\ \ast & 0 & z_{22}& 1 & 0 & 0\\
z_{11} & 0 & z_{12} & 0 & z_{13} & 1
\end{array}\right]$ & \ydiagram{1 + 1, 3} &$\{3, 1\}$\\[15pt]

$\{2, 3\}$ & $\{ 1\} $ & $\left[
\begin{array}{ccc|ccc}
1 & 0 & 0 & 0 & 0 & 0\\ 0 & \ast & z_{23}& 1 & 0 & 0\\
0 & z_{11} & z_{12} & 0 & 1 & 0
\end{array}\right]$ & \ydiagram{1 + 1, 2} &$\{2, 1\}$\\[15pt]

$\{1\}$ & $\{2, 3\} $ & $\left[
\begin{array}{ccc|ccc}
\ast & 1 & 0 & 0 & 0 & 0\\ \ast &0 & 1& 0 & 0 & 0\\
z_{11} & 0 & 0 & z_{12} & z_{13} & 1
\end{array}\right]$ & \ydiagram{3} &$\{3\}$\\[15pt]

$\{2\}$ & $\{1, 3\} $ & $\left[
\begin{array}{ccc|ccc}
1 & 0 & 0& 0 & 0 & 0\\ 0 & \ast & 1& 0 & 0 & 0\\
0 & z_{11} & 0 & z_{12} & 1 & 0
\end{array}\right]$ & \ydiagram{2} &$\{2\}$\\[15pt]

$\{3\}$ & $\{1, 2\}$ & $\left[
\begin{array}{ccc|ccc}
1 & 0 & 0 & 0& 0 & 0\\ 0  & 1 & 0 & 0 & 0 & 0\\
0 & 0 & z_{11} & 1 & 0 & 0
\end{array}\right]$ & \ydiagram{1} &$\{1\}$\\[15pt]

$\phi$ & $\{1, 2, 3\}$ & $\left[
\begin{array}{ccc|ccc}
1 & 0 & 0& 0 & 0 & 0 \\ 0 & 1 & 0 & 0 & 0 & 0 \\ 0 & 0 & 1 & 0 & 0 & 0
\end{array}\right]$ & $\phi$ &$\phi$\\[15pt]

\hline
\end{tabular}
\label{tab:LagCR6}
\caption{The eight complex Schubert cells of $\Lag^{\C}(\R^6)$}
\end{table}The integral homology groups are then
\[ H_{k}(\Lag^\C(\R^6); \Z) \cong \begin{cases} \Z & \text{ if } k = 0, 2, 4, 8, 10, 12\\
\Z^2 & \text{ if } k = 6\\
0 & \text{ otherwise.} \end{cases}\]
For $\Lag^\R(\R^6)$ the CW structure is shown in Figure \ref{fig:LagRR6}.
\begin{figure}
\centering \includegraphics[width = 0.6\textwidth]{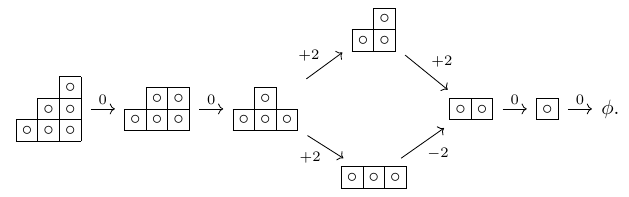}
\caption{A diagrammatic description of the CW structure of $\Lag^{\R}(\R^6)$ by real Schubert cells. Sign conventions of attaching maps differ from \protect\cite{Rabelo} Example 4.3.}
\label{fig:LagRR6}
\end{figure}
The integral homology groups are then
\[ H_{k}(\Lag^\R(\R^6); \Z) \cong \begin{cases} \Z & \text{ if } k = 0, 1, 5, 6\\
\Z_2 & \text{ if } k = 2, 3\\
0 & \text{ otherwise.} \end{cases}\]
\end{example}

%% file: Mixed.tex
\section{Schubert cells of mixed type}
\label{sec:mixed}

\subsection{Shifted Young diagrams of Schubert cells of mixed type}
\label{subsec:SYDmixed}

Identify $\C_{(k, \ell)}$ with $\R_{(k, \ell)} \times i \R_{(k, \ell)}$ with coordinates $z_{k\ell} = x_{k\ell} + i y_{k\ell}$. We give the lexicographical order on the Cartesian product $\lambda \times \{x, y\}$, from the lexicographical order on $\lambda$, and $x < y$. So $((k, \ell), t) \le ((k', \ell'), t')$ if either $(k, \ell) < (k', \ell')$ or both $(k, \ell) = (k', \ell')$ and $t \le t'$. Then we identify $x_{k\ell}$ with $((k, \ell), x)$ and $y_{k\ell}$ with $((k, \ell), y)$.

Let
\[ \mathring{\mu} := \bigcap_{\substack{\mu' \le \mu \\ |\mu'| = |\mu| - 1}} \mu'.\]
and
$C_{\lambda+i\mu}:= (z^{\lambda})^{-1} (U_{\lambda + i \mu})$ where
\[ U_{\lambda + i \mu}:= \prod_{(k, \ell) \in \mathring{\mu}} \C_{(k, \ell)} \times \prod_{(k, \ell) \in \mu \setminus \mathring{\mu}} (\R_{(k, \ell)}\times i \R_{(k, \ell)}^\times) \times \prod_{(k, \ell) \in \lambda \setminus \mu} \R_{(k, \ell)}.\]

Let $\mathcal{E}(\mu):= \{\varepsilon: \mu \setminus \mathring{\mu} \to \{+1, -1\}\}$, and let
\[ U_{\lambda + i \mu, \varepsilon}:= \prod_{(k, \ell) \in \mathring{\mu}} \C_{(k, \ell)} \times \prod_{(k, \ell) \in \mu \setminus \mathring{\mu}} (\R_{(k, \ell)}\times \varepsilon(k, \ell) i\R_{(k, \ell)}^{+}) \times \prod_{(k, \ell) \in \lambda \setminus \mu} \R_{(k, \ell)}\]
and $C_{\lambda +i \mu, \varepsilon} := (z^{\lambda})^{-1} (U_{\lambda + i \mu, \varepsilon})$. Then $C_{\lambda + i \mu, \varepsilon}$ are the connected components of $C_{\lambda + i \mu}$.

The shifted Young diagrams of mixed type $\lambda + i \mu$ can be decorated by adding labels on $\lambda \setminus \mathring{\mu}$. We will label copies of $\C$ with $\ytableaushort{\,}$, copies of $\R \times i \R^\times$ with $\ytableaushort{\times}$, copies of $\R$ with $\ytableaushort{\circ}$, copies of $\R \times i \R^+$ by $\ytableaushort{+}$, and copies of $\R \times -i \R^+$ by $\ytableaushort{-}$.

\begin{example}
\ytableausetup{baseline, boxsize = 7pt}
Suppose
\[\lambda = \begin{ytableau} \none & \none & \\ \none &&\\ && \end{ytableau}, \quad \mu = \begin{ytableau} \none & \none & 
\none \\ \none & & \none\\ && \end{ytableau}. \]
The shifted Young diagram of $\lambda + i \mu$ is denoted as
\[\begin{ytableau} \none & \none & \circ \\ \none & \times & \circ \\  && \times \end{ytableau}.\]
The shifted Young diagrams of the four connected components of $C_{\lambda + i \mu}$ are
\[\begin{ytableau} \none & \none & \circ \\ \none & + & \circ \\  && + \end{ytableau}, \quad \begin{ytableau} \none & \none & \circ \\ \none & + & \circ \\  && -\end{ytableau}, \quad \begin{ytableau} \none & \none & \circ \\ \none & - & \circ \\  && + \end{ytableau}, \quad \begin{ytableau} \none & \none & \circ \\ \none & - & \circ \\  && - \end{ytableau}. \]
\end{example}

\subsection{Attaching maps of Schubert cells of mixed type}
\label{subsec:attachmixed}

\begin{lemma}\label{lem:specialsquare}
  If $\mu \nleq \mu'$, then there exists an element in $\mu \setminus \mathring{\mu}$ not in $\mu'$.
\end{lemma}
\begin{proof}
    $\mu' \cap \mu < \mu$, so there exists a $\mu''$ with $|\mu''| = |\mu| - 1$ such that $\mu' \cap \mu \le \mu'' < \mu$. Then the unique square in $\mu \setminus \mu''$ is contained in $\mu \setminus \mathring{\mu}$ but is not contained in $\mu'$.
\end{proof}

Suppose $\mu \neq \mu'$, and without loss of generality $\mu \nleq \mu'$. Then the square given by Lemma \ref{lem:specialsquare} has the coordinate $(k, \ell)$. Then $y_{k\ell}^{\lambda}(\Gamma)$ is nonzero for $\Gamma \in C_{\lambda + i \mu}$ and zero for $\Gamma \in C_{\lambda + i \mu'}$. So the $C_{\lambda + i \mu}$'s are disjoint. Moreover, $\C^{|\lambda|}$ is partitioned by how many of the last $k$ coordinates are real. So we have the partitions
\[ C_{\lambda} = \bigsqcup_{\mu: \mu \le \lambda} C_{\lambda +i \mu}, \quad \Lag^{\C}(\R^{2n}) = \bigsqcup_{\substack{\lambda, \mu\\ \mu \le \lambda}} C_{\lambda + i \mu}.\]

\begin{lemma}[\label{thm:frontiercellwisemixed}Frontier condition of $C_\lambda$]
The following are equivalent
\begin{enumerate}
    \item $C_{\lambda + i \mu} \cap \overline{C_{\lambda + i \mu'}} \cap C_{\lambda} \neq \phi$.
    \item $C_{\lambda + i \mu} \subseteq \overline{C_{\lambda + i \mu'}} \cap C_{\lambda}$.
    \item $\mu \le \mu'$.
\end{enumerate}
\end{lemma}

\begin{proof}
    $(2) \implies (1)$ is immediate, as $C_{\lambda + i \mu}$ is nonempty. $(3)
    \implies (2)$: $U_{\lambda + i \mu} \subseteq \overline{U_{\lambda + i \mu'}}$, so $(z^{\lambda})^{-1} (U_{\lambda + i \mu}) \subseteq (z^{\lambda})^{-1}(\overline{U_{\lambda + i \mu'}})$.  $(1) \implies (3)$: If $\mu \nleq \mu'$, by Lemma \ref{lem:specialsquare}, there exists a $(k, \ell) \in \mu \setminus \mathring{\mu}$ not in $\mu'$. Then $y_{k\ell}^{\lambda}(\Gamma_k) = 0$ for any sequence $\Gamma_k \in C_{\lambda + i \mu'}$, but $y_{k\ell}^{\lambda}(\Gamma) \neq 0$ for all $\Gamma \in C_{\lambda + i \mu}$. 
\end{proof}

\begin{remark}Immediately, we can verify
$C_{\lambda + i \phi} \cong C_{\lambda}^\R$, $C_{\phi + i \phi} = C_{\phi} = C_{\phi}^\R$, $C_{\lambda + i \lambda}$ is open dense in $C_{\lambda}$, and
$\overline{C_{\lambda + i \mu'}} \cap C_{\lambda} = \bigsqcup_{\mu:\mu \le \mu'} C_{\lambda + i \mu}$.
\end{remark}

\begin{proposition} \label{prop:Jacobianmixed}
If $\lambda \le \lambda'$, $|\lambda'| = |\lambda|+1$, and $\mu \le \lambda$
\[ \varphi_{\lambda}^{\lambda'} ( C_{\lambda + i \mu} \times \R^\times) \subseteq C_{\lambda' + i \mu}.\]

Moreover, suppose $\lambda = \lambda_J$ and $\lambda' = \lambda_{(J')^{\vee}}$, and denote the real Jacobian determinant of the transition map
\[z^{\lambda'} \circ \varphi_{\lambda}^{\lambda'} \circ ((z^{\lambda})^{-1} \times \Id_{\R^\times}) : U_{\lambda + i \mu} \times \R^\times \to U_{\lambda' + i \mu}\]
of $C_{\lambda'+i \mu}$ by $T_{\lambda + i \mu}^{\lambda'+i\mu}$. Then $T_{\lambda + i \mu}^{\lambda' + i \mu}$ is a function of $t$ only, and $T_{\lambda +i \mu}^{\lambda' + i \mu}(t)$ is equal to
\begin{equation} \label{eq:JacobianACRmixed} \begin{cases} (-1)^{|c_\star| + |c_\star \cap \mu|} \cdot t^{-|c_\star| - |c_\star \cap \mu|-1} & \text{ if }  |J'| > |J| \\ (-1)^{|c_\star| + |c_\star \cap \mu| + |\rho_\star| + |\rho_\star \cap \mu|} \cdot t^{-|\alpha_\star| - |\alpha_\star \cap \mu| - |\mu \cap \{(\star, \star)\}| - 2} & \text{ if } |J'| = |J| \end{cases}\end{equation}
where $t$ is the coordinate of $\R^\times$.
\end{proposition}

\begin{proof}
We examine equations \eqref{eq:znew} and \eqref{eq:zsame}. When $(k, \ell) \in \mathring{\mu}$, $z'_{k\ell}$ is complex valued. When $(k, \ell) \in \mu \setminus \mathring{\mu}$, $z'_{k\ell}$ has nonvanishing imaginary part because $z_{k\ell}$ has nonvanishing imaginary part, and if either $k' > k$ or $\ell' > \ell$, $z_{k', \ell'}$ is real valued for $\Gamma \in C_{\lambda + i \mu}$. Similarly, when $(k, \ell) \in \lambda \setminus \mu$, $z'_{k\ell}$ has vanishing imaginary part because $z_{k\ell}$ has vanishing imaginary part, and if either $k' > k$ or $\ell' > \ell$, $z_{k', \ell'}$ is real valued.

The ordering of the variables $\{x_{k\ell}\}_{(k, \ell) \in \lambda} \cup\{y_{k\ell}\}_{(k, \ell) \in \mu}$ ensures that the real Jacobian determinant of the transition map is upper triangular. In equations \eqref{eq:undercolumn}, \eqref{eq:topcolumn}, \eqref{eq:leftarch}, \eqref{eq:toparch}, \eqref{eq:rightarch}, \eqref{eq:leftcornerarch}, and \eqref{eq:rightcornerarch}, we double count contributions of $\mp t^{-1}$ (and $\mp t^{-2}$) when $(k, \ell) \in \mu$. We also double count the sign contribution due to rearranging the order of $x_{\star + j_\star^\vee, \star}$ from $|\rho_\star|$ to $|\rho_\star| + |\rho_\star \cap \mu|$.
\end{proof}

\begin{proposition}[\label{prop:frontiermixed}Frontier condition]
The following are equivalent
\begin{enumerate}
    \item $C_{\lambda + i \mu} \cap \overline{C_{\lambda' + i \mu'}} \neq \phi$.
    \item $C_{\lambda + i\mu} \subseteq \overline{C_{\lambda' + i \mu'}}$.
    \item $\lambda \le \lambda'$ and $\mu \le \mu'$.
\end{enumerate}
\end{proposition}
\begin{proof}
$(2) \implies (1)$ is immediate as $C_{\lambda + i \mu}$ is nonempty. $(3) 
\implies (2)$: There exists an increasing sequence $\lambda_1 \le \lambda_2 \le \cdots \le \lambda_\ell$ such that $\lambda_1 = \lambda$, $\lambda_\ell = \lambda'$, and $|\lambda_{k+1}| = |\lambda_k|+1$ for $1 \le k < \ell = |\lambda'| - |\lambda|$, and a sequence of embeddings $\{\varphi_{\lambda_k}^{\lambda_{k+1}}\}_{k = 1}^\ell$. For each $\Gamma \in C_{\lambda + i \mu}$, let
\[ \Gamma_{t_1, \cdots, t_\ell} := \varphi_{\lambda_{\ell - 1}}^{\lambda_\ell}(\cdots(\varphi_{\lambda_2}^{\lambda_3} (\varphi_{\lambda_1}^{\lambda_2} (\Gamma, t_1), t_2), \cdots), t_\ell) \quad t_1, \cdots,  t_\ell \in \R\]
where if $t_j = 0$ for some $j \in \{1, \cdots, \ell\}$, then $t_k = 0$ for all $k \ge j$. Then by equations \ref{eq:newLagrangian}, \ref{eq:sameLagrangian} and Proposition \ref{prop:Jacobianmixed}, $\Gamma_{t_1, \cdots, t_\ell} \in C_{\lambda' + i \mu} \subseteq \overline{C_{\lambda' + i \mu'}}$, and
\[ \lim_{t_1, \cdots, t_\ell \to 0} \Gamma_{t_1, \cdots, t_\ell} = \Gamma.\]
$(1) \implies (3)$: If $\lambda \nleq \lambda'$, then $C_\lambda \cap \overline{C_{\lambda'}} = \phi$, which contradicts the assumption. So $\lambda \le \lambda'$. By assumption there exists a sequence $\Gamma'_m \to \Gamma$ with $\Gamma'_m \in C_{\lambda' + i\mu'}$. By the computation of Jacobian determinants of equation \eqref{eq:Jacobian},  $\varphi_{\lambda_k}^{\lambda_{k+1}}$ are local homeomorphisms from $C_{\lambda_k + i \mu'} \times \R^\times$ to $C_{\lambda_{k+1} + i \mu'}$. By induction, we get a local homeomorphism $C_{\lambda + i \mu'} \times (\R^\times)^{|\lambda'| - |\lambda|}$ to $C_{\lambda' + i \mu'}$. Let $(\Gamma_m, (t^m_1, \cdots, t^m_\ell))$ be the image $\Gamma_m'$ under these local homeomorphisms, which exist for $m$ sufficiently large. Then if $\mu \nleq \mu'$, there exists $(a, b) \in \mu \setminus \mathring{\mu}$ not in $\mu'$. Then $y_{ab}^\lambda(\Gamma_m) = 0$ but $y_{ab}^\lambda(\Gamma) \neq 0$, which is a contradiction. So $\mu \le \mu'$.
\end{proof}

We obtain attaching maps of Schubert cells of mixed type by restricting the attaching maps of Proposition \ref{prop:attachingmaps}.

\begin{theorem}[\label{thm:degreesmixed}Attaching maps of Schubert cells of mixed type]~
If $|\lambda| + |\mu| + 1 = |\lambda'| + |\mu'| + 1$, let
\[ g_{\lambda + i \mu, \varepsilon}^{\lambda' + i \mu', \varepsilon'} : \partial C_{\lambda' +i \mu', \varepsilon'} \to \overline{C_{\lambda + i \mu, \varepsilon}}.\]
\begin{enumerate}

\item If $\lambda \nleq \lambda'$ or $\mu \nleq \mu'$, then
\begin{equation} \label{eq:bdy1}
\deg g_{\lambda +i \mu, \varepsilon}^{\lambda' + i \mu', \varepsilon'} = 0.
\end{equation}

\item If $\mu = \mu'$ and $\lambda \le \lambda'$ let
\[\varepsilon^{\pm}(k, \ell) :=\begin{cases} \pm \varepsilon(k, \ell) & \text{ if } (k, \ell) \in ((\mu \setminus \mathring{\mu}) \cap (\alpha_\star \setminus c_\star))\setminus \{ (\star, \star)\}\\
\mp \varepsilon(k, \ell) & \text{ if } (k, \ell) \in (\mu \setminus \mathring{\mu}) \cap c_\star\\
\varepsilon(k, \ell) & \text{ otherwise,} \end{cases}\]
where $\alpha_\star$ is empty if $\lambda'$ has more rows than $\lambda$.
Moreover, let
\[\delta^{\pm}(\varepsilon, \varepsilon') := \begin{cases} 1 & \text{ if } \varepsilon' = \varepsilon^{\pm}\\
0 & \text{ otherwise.} \end{cases}
\]
Then
\begin{equation}
\label{eq:bdy2}
\deg g_{\lambda +i \mu, \varepsilon}^{\lambda' + i \mu, \varepsilon' } = T_{\lambda +i \mu}^{\lambda' + i \mu} (+1) \cdot \delta^+(\varepsilon, \varepsilon') - T_{\lambda +i \mu}^{\lambda' + i \mu} (-1)\cdot \delta^-(\varepsilon, \varepsilon'),
\end{equation}
where $T_{\lambda + i \mu}^{\lambda' + i\mu}$ is as in equation \eqref{eq:JacobianACRmixed}.

\item If $\lambda = \lambda'$ and $\mu \le \mu'$ then
\begin{equation} \label{eq:bdy3}
\deg g_{\lambda + i \mu, \varepsilon}^{\lambda' + i \mu', \varepsilon'} = \begin{cases} - (-1)^{|\lambda \setminus \mu'|}\cdot \varepsilon'(\star_\mu) & \text{ if } \varepsilon' = \varepsilon \text{ on } (\mu' \setminus \mathring{\mu}') \cap (\mu \setminus \mathring{\mu})\\0 & \text{ otherwise.} \end{cases} \end{equation}
where $\star_\mu$ is the unique element of $\mu' \setminus \mu$.

\end{enumerate}
\end{theorem}

\begin{proof}
If $\lambda \nleq \lambda'$ or $\mu \nleq \mu'$, then the degree of $g_{\lambda + i \mu, \varepsilon}^{\lambda' + i \mu', \varepsilon'}$ is zero by the frontier condition.

If $\mu = \mu'$ and $\lambda \le \lambda'$, we look at Equations \eqref{eq:znew} and \eqref{eq:zsame}. When $(k, \ell) \in \mathring{\mu}$, both $z_{k\ell}$ and $z_{k\ell}'$ are complex valued with no restrictions, and when $(k, \ell) \in \lambda \setminus \mu$, both $z_{k\ell}$ and $z_{k\ell}'$ are real valued with no restrictions. When $(k, \ell) \in \mu \setminus \mathring{\mu}$, the coefficient of $z_{k\ell}$ is $t^{-1}$ if $(k, \ell) \in (\alpha_\star \setminus c_\star)\setminus \{(\star, \star)\}$ and $-t^{-1}$ if $(k, \ell) \in c_\star$ ($(\lambda' \setminus \lambda)\cap (\mu \setminus \mathring{\mu})$ is empty), and either $1$ or $t^{-2}$ otherwise. So when $t>0$, $\varphi_{\lambda}^{\lambda'}(C_{\lambda +i \mu, \varepsilon)} \times \R^+) \subseteq C_{\lambda' + i \mu, \varepsilon'}$ if and only if $\varepsilon' = \varepsilon^+$, and if $t < 0$, $\varphi_{\lambda}^{\lambda'}(C_{\lambda +i \mu, \varepsilon} \times \R^-) \subseteq C_{\lambda' + i \mu, \varepsilon'}$ if and only if $\varepsilon' = \varepsilon^-$.

If $\lambda = \lambda'$ and $\mu \le \mu'$, then $-\varepsilon(\star_\mu)$ points in the gradient direction of $y_{\star_\mu}$, and the sign correction due to the ordering of $y_{k\ell}$ is $(-1)^{|\lambda \setminus \mu'|}$.

\end{proof}
\begin{example}
    Suppose $\lambda = \lambda_{\{4, 2\}}$, $\lambda' = \lambda_{\{ 4, 3\}}$, and $\mu = \lambda_{\{ 4, 2\}}$. Then
    \[ \lambda + i \mu = \ytableaushort{ \none \, \times \none, \, \, \, \times} \text{ and } \lambda' + i \mu = \ytableaushort{ \none \, \times {*(lightgray)\circ}, \, \, \, \times}.\]
Then $\mu \setminus \mathring{\mu} = \{ (1, 4), (2, 3)\}$,
\[ ((\mu \setminus \mathring{\mu}) \cap (\alpha_\star \setminus c_\star))\setminus \{(\star, \star)\} = \{ (2, 3)\}\]
and
\[ (\mu \setminus \mathring{\mu}) \cap c_\star = \{ (1, 4)\}.\]
Denote $\varepsilon$ by an ordered pair $(\varepsilon(1, 4), \varepsilon(2, 3))$. Since there are no restrictions on the values of $z_{11}'$, $z_{12}'$, $z_{13}'$, $z_{22}'$, and 
\begin{eqnarray*}
    z_{14}' &=& -t^{-1} z_{14}\\
    z_{23}' &=& t^{-1} z_{23},
\end{eqnarray*}
so 
\begin{eqnarray*}
\varphi_{\lambda}^{\lambda'} (C_{\lambda + i \mu, (\pm 1, +1)} \times \R^+) &\subseteq& C_{\lambda' + i \mu, (\mp 1, +1)}\\
\varphi_{\lambda}^{\lambda'}(C_{\lambda + i \mu, (\pm 1, -1)} \times \R^+) &\subseteq& C_{\lambda' + i \mu, (\mp 1, -1)}
\end{eqnarray*}
and attach with degree $T_{\lambda + i \mu}^{\lambda' + i \mu} (+1) = - (+1)^{-9} = -1$, and
\begin{eqnarray*}
\varphi_{\lambda}^{\lambda'}(C_{\lambda + i \mu, (+1, \pm 1)} \times \R^-) &\subseteq&C_{\lambda' + i \mu, (+ 1, \mp 1)}\\
\varphi_{\lambda}^{\lambda'} (C_{\lambda + i \mu, (- 1, \pm 1)} \times \R^-) &\subseteq& C_{\lambda' + i \mu, (-1, \mp 1)}
\end{eqnarray*}
and attach with degree $-T_{\lambda + i \mu}^{\lambda' + i \mu} (-1) = - (-1)^{-9} = +1$. The attaching maps are shown in Figure \ref{fig:bdy2}.

\begin{figure}
\centering \includegraphics{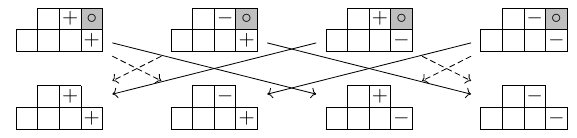}
\caption{A diagrammatic description of the attaching maps $\partial C_{\lambda_{\{4, 3\}} + i\lambda_{\{4, 2\}}} \to C_{ \lambda_{\{4, 2\}} + i\lambda_{\{4, 2\}}}$. Arrows represent attaching maps of degree $+1$ and dashed arrows represent attaching maps of degree $-1$.}
\label{fig:bdy2}
\end{figure}

\end{example}

\begin{example}
Suppose $\lambda = \lambda' = \lambda_{\{4, 3\}}$, $\mu' = \lambda_{\{ 4, 2\}}$, $\mu = \lambda_{\{ 4, 1\}}$, so that
\[ \lambda + i \mu' = \ytableaushort{ \none \, {*(lightgray)\times} \circ, \, \, \, \times}\text{ and } \lambda + i\mu = \ytableaushort{\none \times {*(lightgray) \circ} \circ, \, \, \, \times}. \]
The attaching maps are shown in Figure \ref{fig:bdy3}.

\begin{figure}
\centering \includegraphics{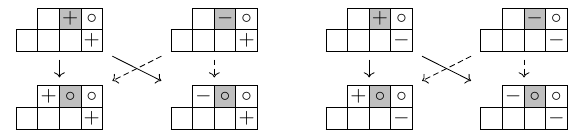}
\caption{A diagrammatic description of the attaching maps $\partial C_{\lambda_{\{4, 3\}} + i\lambda_{\{4, 2\}}} \to C_{ \lambda_{\{4, 3\}} + i\lambda_{\{4, 1\}}}$. Arrows represent attaching maps of degree $+1$ and dashed arrows represent attaching maps of degree $-1$.}
\label{fig:bdy3}
\end{figure}

\end{example}

%% file: Applications.tex
\section{Applications}
\label{sec:applications}

If $n, m$ are positive integers, let $\iota_{n, m}: \C^{2n} \hookrightarrow \C^{2(n+m)}$ be an embedding given by the linear extension of $\ee_{i} \mapsto \ee_{i + m}$, $\ff_{i} \mapsto \ff_{i + m}$. $\iota_{n, m}$ induces an embedding of complex Lagrangian Grassmannians $I_{n, m}: \Lag^\C(\R^{2n}) \hookrightarrow \Lag^\C(\R^{2(n+m)})$ as
\[ I_{n, m}(\Gamma) := \span_\C\{\ee_1, \cdots, \ee_m\} \oplus \iota_{n, m}(\Gamma).\]

If $M_\Gamma$ is the $n \times 2n$ matrix representing the reduced row echelon form of $\Gamma \in \Lag^\C(\R^{2n})$, then the $(n+m) \times 2(n+m)$ matrix representing the reduced row echelon form of $I_{n, m}(\Gamma)$ is 
\[ \begin{bmatrix} \operatorname{Id}_{m, m} & 0_{m, 2n} & 0_{m, m}\\
0_{n, m} & M_{\Gamma} & 0_{m, m}\end{bmatrix}.\]
$I_{n ,m}$ maps each Schubert cell of mixed type $C_{\lambda  + i \mu, \varepsilon}$ of $\Lag^\C(\R^{2n})$ to the Schubert cell of mixed type $C_{\lambda + i \mu, \varepsilon}$ of $\Lag^\C(\R^{2(n+m)})$. $I_{n, m}$ is continuous, so it preserves all incidence relations. Moreover, $\varphi_\lambda^{\lambda'} \circ I_{n, m} = I_{n, m} \circ \varphi_\lambda^{\lambda'}$, so it preserves all degrees of attaching maps as well. So the $I_{m, n}(\Lag^{\C}(\R^{2n}))$ can be realized as a subcomplex of $\Lag^\C(\R^{2(n+m)})$.

Similarly, let $\iota_{m, n}^\R: \R^{2m} \hookrightarrow \R^{2n}$ and $I_{n, m}^{\R}: \Lag^\R(\R^{2m}) \hookrightarrow \Lag^\R(\R^{2n})$ be the corresponding maps for real coefficients. Suppressing these identifications, we will regard $\Lag^\R(\R^{2n})$ as subcomplexes of $\Lag^\C(\R^{2(n+m)})$. 

\begin{corollary}[Homotopy extension property]\label{cor:HEP}
$(\Lag^\C(\R^{2(n+m)}), \Lag^\R(\R^{2n}))$ satisfies the homotopy extension property.
\end{corollary}

\begin{corollary} \label{cor:nullhomotopic}
If $0 < k \le n$ and $k$ is even then $\overline{C}_{\lambda_{\{k\}}}^\R$ defines a nontrivial torsion class in $H_k(\Lag^\R(\R^{2n}); \Z)$, and is homotopic to a $2k$-sphere inside $\Lag^\C(\R^{2n})$. 
\end{corollary}

\begin{proof}
By equation \eqref{eq:bdyreal}, $\overline{C}_{\lambda_{\{k\}} + i \phi}$ defines a homology class in $H_k(\Lag^\R(\R^{2n}); \Z)$.
By induction on $k$, $\overline{C}_{\lambda_{\{k-1\}} + i \lambda_{\{ 1\}}, +}$ is contractible. Taking the quotient of this subcomplex inside $\overline{C}_{\lambda_{\{k\}} + i \lambda_{\{ 1\}}, +}$ we get a homotopy equivalence between $\overline{C}_{\lambda_{\{k\}} + i \lambda_{\{ 1\}}, +}$ and the closed $2k+1$-ball. This equivalence identifies the subcomplex $\overline{C}_{\lambda_{\{k\}} + i \phi}$ with its boundary.
\end{proof}

\begin{remark}[$k = 2$ case]
If $k = 2 \le n$, $\overline{C}_{\lambda_{\{2\}}}^\R$ represents the generator of $H_2(\Lag^\R(\R^{2n}); \Z) \cong \Z_2$. Topologically, this set is homeomorphic to a pinched torus. One way to see this is by doing row reductions to
\[ \left[\begin{array}{cc|cc}  \cos \theta & \sin \theta & 0 & 0 \\
-\cos \psi \sin \theta & \cos \psi \cos \theta & -\sin \psi \cos \theta & \sin \psi \sin \theta \end{array}\right] \in E(\lambda_{\{2\}})\]
at various values of $\theta, \psi \in [0, \pi]$ (cf Liu \cite{LeiLiu}). So the generator of $H_2(\Lag^\R(\R^{2n}); \Z)$ is spherical, which is something we cannot conclude from the Hurewicz theorem. By Corollary \ref{cor:nullhomotopic} this pinched torus is homotopic to a $2$-sphere inside $\Lag^\C(\R^{2n})$.
\end{remark}